\documentclass[UTF-8,reqno]{amsart}
\usepackage{enumerate}
\usepackage{mhequ}
\usepackage[margin=1.2in]{geometry}
\usepackage{amssymb,url,color, booktabs,nccmath}
\usepackage{mathrsfs}
\usepackage{enumitem}
\usepackage{graphicx}
\usepackage{tikz}
\usetikzlibrary{shapes,snakes}
\usetikzlibrary{calc}
\usetikzlibrary{decorations.shapes}

\usepackage[colorinlistoftodos,prependcaption,color=yellow,textsize=tiny,textwidth=2cm]{todonotes}
\usepackage{color}
\usepackage[colorlinks=true]{hyperref}
\hypersetup{
    linkcolor=blue,          
    citecolor=red,        
    filecolor=blue,      
    urlcolor=cyan
}

\newcommand{\typocheck}[2]{%
  \texorpdfstring{%
    \ifmmode
      \text{\sout{\ensuremath{#1}}}\,\textcolor{red}{#2}%
    \else
      \sout{#1}\,\textcolor{red}{#2}%
    \fi
  }{#2}%
}

\definecolor{darkergreen}{rgb}{0.0, 0.5, 0.0}
\definecolor{darkblue}{RGB}{0,0,139}

\numberwithin{equation}{section}

\newcommand{\be}{\begin{eqnarray}}
\newcommand{\ee}{\end{eqnarray}}
\newcommand{\ce}{\begin{eqnarray*}}
\newcommand{\de}{\end{eqnarray*}}
\newtheorem{theorem}{Theorem}[section]
\newtheorem{lemma}[theorem]{Lemma}

\newtheorem{proposition}[theorem]{Proposition}

\allowdisplaybreaks

\tikzset{
        dot/.style={circle,fill=black,inner sep=0pt, outer sep=0.7pt, minimum size=1mm},
        Phi/.style={white!40!red,thick,snake=coil,segment amplitude=0.6pt, segment length=2pt},
         Z/.style={black!40!green,thick,snake=coil,segment amplitude=0.6pt, segment length=2pt},
        C/.style={thick,black!20!blue},
          Cr/.style={thick,black!20!red},
            Cg/.style={thick,black!20!green},
       }

\usepackage{mathtools,amssymb,amsfonts,amsthm,amscd,amsmath,amsgen,upref,enumerate,nicefrac,pdfpages}

\begin{document}
\newcommand{\Q}{\mathbb{Q}}
\newcommand{\R}{\mathbb{R}}
\newcommand{\DD}{\mathbb{D}}
\newcommand{\cD}{\mathcal{D}}
\newcommand{\cG}{\mathcal{G}}
\newcommand{\TND}{\mathbb{T}^d_N}
\newcommand{\TD}{\mathbb{T}^d}
\newcommand{\I}{\mathbb{I}}
\newcommand{\Z}{\mathbb{Z}}
\newcommand{\N}{\mathbb{N}}
\newcommand{\F}{\mathcal{F}}
\newcommand{\p}{\mathbb{P}}
\newcommand{\hs}{\hspace{1cm}}
\newcommand{\XX}{\mathbb{X}}
\newcommand{\M}{\mathcal{M}}
\newcommand{\Pp}{\mathcal{P}}
\newcommand{\W}{\mathcal{W}}
\newcommand{\ED}[1]{{#1}^\epsilon_\delta}
\newcommand{\EDM}[1]{{#1}^\epsilon_{\delta,m}}
\newcommand{\EOM}[1]{{#1}^\epsilon_{0,m}}
\newcommand{\EDMg}[1]{{#1}^{\epsilon,g}_{\delta,m}}
\newcommand{\RR}{\mathcal{R}}
\newcommand{\E}{\mathbb{E}}
\newcommand{\h}{\mathcal{H}}
\newcommand{\cL}{\mathcal{L}}
\newcommand{\Ii}{\mathrm{I}}
\newcommand{\eps}{\epsilon}
\newcommand{\supp}{\mathrm{supp}}
\newcommand{\law}{\mathrm{Law}}
\newcommand{\leb}{\lambda}
\newcommand{\id}{\mathrm{id}}
\newcommand{\pr}{\mathrm{pr}}
\newcommand{\Ss}{\mathfrak{S}}
\newcommand{\inter}{\mathrm{int\,}}
\newcommand{\dhcomment}[1]{\textbf{\textcolor{red}{#1}}}
\newcommand{\PP}{\mathbb{P}}
\newcommand{\cH}{\mathcal{H}}
\newcommand{\cX}{\mathcal{X}}
\newcommand{\cE}{\mathcal{E}}
\newcommand{\T}{\mathbb T}
\newcommand{\scrE}{\mathscr{E}}
\newcommand{\step}[1]{\paragraph{\textbf{Step #1}}}
\newcommand{\Law}{{\rm Law}}

\newcommand{\Ent}{\textrm{Ent}}

\newcommand{\norm}[1]{\| #1 \|}
\newcommand{\abs}[1]{|#1|}
\newcommand{\ud}[1]{\, \mathrm{d} #1}
\newcommand{\dx}{\ud{x}}
\newcommand{\dy}{\ud{y}}
\newcommand{\dxi}{\ud \xi}
\newcommand{\deta}{\ud{\eta}}
\newcommand{\dr}{\ud{r}}
\newcommand{\drp}{\ud{r'}}
\newcommand{\dxp}{\ud{x'}}
\newcommand{\dxip}{\ud{\xi'}}
\newcommand{\dyp}{\ud{y'}}
\newcommand{\ds}{\ud{s}}
\newcommand{\dt}{\ud{t}}
\newcommand{\dz}{\ud{z}}
\newcommand{\dd}{d}
\newcommand{\C}{\textrm{C}}
\newcommand{\ve}{\varepsilon}
\newcommand{\sgn}{\textrm{sgn}}
\newcommand{\mcS}{\mathcal{S}}
\newcommand{\Supp}{\textrm{Supp}}

\renewcommand{\d}{\delta}

\title[LDP for the Wick ordered cubic NLW]{\Large Low-Temperature Large Deviations for the Two-Dimensional Wick-Ordered Cubic Wave Equation}

\author{Daniel Heydecker}

\address{Department of Mathematics, Postboks 1053, Blindern, 0316 Oslo}
	\email{daniehey@uio.no}



\begin{abstract}
We prove a trajectory large deviation principle, in the low-temperature limit, for the two-dimensional Wick-ordered defocusing cubic nonlinear wave equation, with initial data drawn from the invariant Gibbs measure. We further show that asymptotically vanishing high-frequency perturbations of the Gibbs initial data, with the same static rate function, generate a continuum of distinct mass-shifted trajectory rate functions.

\end{abstract}

\maketitle

\setcounter{tocdepth}{1}
\tableofcontents

\section{Introduction \& Statement of Results}
Among the most important results of $20^{\rm th}$ century constructive quantum field theory is the construction of the two-dimensional $P(\phi)_2$ {Euclidean field} \cite{Simon,GlimmJaffe,nelson1966quartic,glimm1968boson,glimm1974wightman,glimm1976convergent}. Adjoining an independent Gaussian velocity gives the corresponding phase-space Gibbs measure, which at inverse temperature $\beta=\eps^{-1}>0$ may be formally written \begin{equation}\label{eq:gibbs-intro}\begin{split}
\mu^\eps(\dd u, \dd v) & =`` \exp\bigg(-\eps^{-1}\underbrace{\int_{\T^2} (\frac12|\nabla u|^2+\frac12 u^2+\frac12v^2+\frac14u^4)}_{:=\cE(u,v)} \bigg)du dv ''. \end{split} 
\end{equation} The Hamiltonian flow associated to the real-valued, quartic energy $\cE$ is the cubic, nonlinear wave equation \begin{equation}\label{eq:classical-nlw}
\partial_t^2u+(1-\Delta)u+u^3=0, \qquad (u,\partial_tu)|_{t=0}=(f,g)\in H^s(\T^2)\times H^{s-1}(\T^2)=:\cH^s, 
\end{equation} which, in $d=2$, is locally well-posed as soon as $s\ge \frac14$\cite{lindblad1995existence,kapitanski1994weak}, and ill-posed for $s<\frac14$ \cite{ChristCollianderTaoNLW},\cite[Section~1.6]{OOT}. For a measure arising in quantum field theory, the passage to the classical regime may be understood through the {\em semiclassical} or {low-temperature} limit \cite{BarashkovThesis,KloseMayorcas,LacoinRhodesVargas,gess2024low}, identifying $\eps\sim \hbar$ on the order of Planck's constant, or as the temperature respectively, and quantifying the concentration on classical minimisers of the action. Mathematically, we quantify the {fluctuations} via a large deviation principle, or equivalently Laplace principle, in the limit $\eps\to 0$. The task of this paper is to investigate the trajectorial LDP for the solution $U^\eps$ to the wave equation in equilibrium $$ \mu^\eps(U^\eps \in A)\asymp \exp(-\eps^{-1}\inf_A J(u)). $$ In order to give a mathematical meaning to both the initial distribution $X^\eps\sim\mu^\eps$ and the corresponding solution $U^\eps$, the cubic and quartic terms appearing in equations \eqref{eq:gibbs-intro} - \eqref{eq:classical-nlw} must be renormalised. Indeed, interpreting the {quadratic} terms of \eqref{eq:gibbs-intro} as defining a Gaussian measure $\gamma^\eps(du, dv)$, the measure is supported only on $\cH^{0-}=\cap_{s<0}\cH^s$,  below the  deterministic local well-posedness threshold $s=\frac14$ for \eqref{eq:classical-nlw}. Nelson's \cite{nelson1966quartic} construction of the measure $\mu^\eps$ is to replace the quartic power by the Wick power \begin{equation}\label{eq:gibbs_intro_2}  \mu^\eps(du,dv):=Z_\eps^{-1}\exp\left(
 -\frac1{4\eps}\int_{\T^2}:u^4:_\eps\,\dd x
 \right)\gamma^\eps(\dd u, \dd v) \end{equation} where the Wick product is defined relative to $\gamma^\eps$. Similarly, to avoid triviality \cite{OhPocovnicuTzvetkov,oh2020remark}, the cube in \eqref{eq:classical-nlw} must be replaced by the third Wick power, giving the Wick-ordered nonlinear wave equation (NLW)
\begin{equation}\label{eq:wick-nlw}
 \partial_t^2u+(1-\Delta)u+:u^3:_\eps=0,
 \qquad (u,\partial_tu)|_{t=0}=X^\eps.
\end{equation}
 For $\mu^\eps$-almost every initial datum, \eqref{eq:wick-nlw} admits a globally defined canonical solution $U^\eps$, and the measure $\mu^\eps$ given by \eqref{eq:gibbs_intro_2} is invariant under the flow \cite{OhThomann,OOT}.  \\\\The static semiclassical limit is classical: For any fixed $s<0$, the laws $\mu^\eps$ satisfy a large deviation principle on $\cH^s$ with speed $\eps^{-1}$ and good rate function given by the Hamiltonian
\begin{equation}\label{eq:energy}
 \cE(f,g)=\frac12\norm f_{H^1}^2+\frac12\norm g_{L^2}^2
             +\frac14\int_{\T^2}f^4\,\dd x,
\end{equation}
understood to be infinite outside $H^1\times L^2$.  {See also \cite{barashkov2023variational,gess2024low} for related static and sharper low-temperature results for the $\Phi^4_2$ measure.} By analogy to Schilder's Theorem, although typical samples of \eqref{eq:gibbs-intro} are distributions, its low-temperature large deviations are governed by the unrenormalised classical energy \eqref{eq:energy}. It is then natural to ask whether the same classicalisation occurs dynamically. For $T>0$ and $u\in \cH^1$, let $\Phi_T(u)$ denote the global energy solution to \eqref{eq:classical-nlw} on $[-T,T]$ with initial data $u$, and set
\[
 \cX_T^s=C([-T,T];\cH^s).
\]
Identifying a scalar solution with the phase-space path $(u,\partial_tu)$, define a candidate rate function $J:\cX_T^s\to[0,\infty]$ by
\begin{equation}\label{eq:dynamic-rate}
 J(u)=
 \begin{cases}
  \cE(u(0)),&u(0)\in H^1\times L^2,\ \text{and }u=\Phi_T(u(0));\\
  +\infty,&\text{otherwise.}
 \end{cases}
\end{equation}
If the singular data-to-solution map for \eqref{eq:wick-nlw} behaved continuously in the topology of the static LDP, the contraction principle would immediately predict \eqref{eq:dynamic-rate}.  The first result of the work is that this prediction is indeed correct.
\begin{theorem}[Trajectory LDP for Gibbs Measure Initial Data]\label{thm:main}
For every $\eps>0$, let $X^\eps\sim\mu^\eps$, and let $U^\eps$ be the canonical solution to \eqref{eq:wick-nlw}.  Then the laws of $U^\eps$ satisfy an LDP on $\cX_T^s$, with speed $\eps^{-1}$ and good rate function $J$.  Thus, for every closed $\mathcal A\subset\cX_T^s$ and every open $\mathcal U\subset\cX_T^s$,
\begin{equation}\label{eq: UB statement}
 \limsup_{\eps\to0}\eps\log\mu^\eps(U^\eps\in\mathcal A)
 \le -\inf_{u\in\mathcal A}J(u),
\end{equation}
and
\begin{equation}\label{eq: lb statement}
 \liminf_{\eps\to0}\eps\log\mu^\eps(U^\eps\in\mathcal U)
 \ge -\inf_{u\in\mathcal U}J(u).
\end{equation}
\end{theorem}

The contraction-principle heuristic motivating the definition of $J$ takes place across the singularity of the model.  As discussed above, the initial data $X^\eps$ are valued in the space $\cH^{0-}$, and in particular the Wick nonlinearity in \eqref{eq:wick-nlw} is not a function of the initial distribution alone in any locally continuous sense.  Indeed, the canonical data-to-solution map is almost everywhere discontinuous in the ambient topology of $\cH^{0-}$ \cite[Theorem~1.11]{OOT}; see also {\cite{ChristCollianderTaoNLW}}. Nor can the continuity of $u\mapsto \Phi_T(u)$ be exploited on the sublevel set $\{\cE<\infty\}=\cH^1$, as the contraction principle requires the continuity of the solution map in the {ambient topology}. {The} resolution follows the strategy of Hairer--Weber \cite[Sections~3--4]{HairerWeber}: first establish a large deviation principle for an enhanced Gaussian object, and then contract through a solution map which is locally continuous on a neighbourhood of the effective domain. In the present setting the enhancement is finite and consists of the linear wave together with its Wick powers.\\\\ Theorem \ref{thm:main} raises a second, more structural question, of whether the dynamical LDP is determined by the static LDP of $\mu^\eps$. We resolve this in the negative by the following theorem.

\begin{theorem}[Failure of the Trajectory LDP for nearly-Gibbs initial data]\label{thm: bad LDP}
There exists a family $\{\overline\mu^\eps\}_{\eps>0}\subset\mathcal P(\cH^s)$ which satisfies the same static LDP as $\mu^\eps$, with speed $\eps^{-1}$ and rate function $\cE$, and such that $\overline\mu^\eps\ll\gamma^\eps$.  However, the canonical solutions $\overline{U}^\eps$ to \eqref{eq:wick-nlw} with initial data $\overline{X}^\eps\sim \overline\mu^\eps$ satisfy a good LDP on $\cX_T^s$ with a rate function different from \eqref{eq:dynamic-rate}.
\end{theorem}

In fact, Section~\ref{sec: pathological} constructs a continuum of such families $\{\overline\mu^{\eps,m}\}_{\eps>0}$, indexed by $m>0$, for which the trajectory rate function is generated by
\begin{equation}\label{eq:shifted-intro}
 \partial_t^2u+(1+m-\Delta)u+u^3=0.
\end{equation}
The construction relies on the perturbation of high-frequency Fourier modes, engineered to produce a perturbation which is $\mathcal{O}(1)$ in $\cH^0$, and hence vanishing in $\cH^{s}$ for any $s<0$. Further, the construction does not rely on a particular choice of which modes are perturbed: The perturbation can be supported in frequency space on any choice $\Lambda_N$ of finite, symmetric sets $\Lambda_N\subset\Z^2$ satisfying
\[
 \inf_{n\in\Lambda_N}|n|\ge N,
 \qquad
 |\Lambda_N|\gtrsim N^\theta, \qquad \theta>0.
\]
  The same construction also allows us to give a short proof of the discontinuity of the data-to-solution map, which we state in Theorem~\ref{prop:instability}.

\subsection{Literature Review and Discussion}

\paragraph*{\em Random-Data Dispersive Equations and Gibbs Dynamics}
The probabilistic study of dispersive equations with rough data and invariant Gibbs measures goes back to Bourgain's construction for the periodic nonlinear Schr\"odinger (NLS) {equation} \cite{Bourgain1994}.  Random-data Cauchy theory for supercritical wave equations was developed by Burq-Tzvetkov \cite{BurqTzvetkov2008I,BurqTzvetkov2008II}; see also \cite{BurqTzvetkov,CollianderOh} for further probabilistic well-posedness results for wave and Schr\"odinger equations.  For the two-dimensional Wick-ordered NLW \eqref{eq:wick-nlw}, the invariant measure and global probabilistic well-posedness were proven by Oh-Thomann \cite{OhThomann}. The viewpoint of the project, both in Theorems \ref{thm:main} - \ref{thm: bad LDP} above and Theorem \ref{prop:instability} below, was motivated by Oh-Okamoto-Tzvetkov \cite{OOT}, who studied the approximation property of the data-to-solution map $X\mapsto U$ by smooth functions. Indeed, an alternative and broadly equivalent approach to Theorem \ref{thm:main} is to show that the solution $U^{\eps,\delta}$ started from $X^{\eps,\delta}:=\kappa_\delta\star X^\eps$, which is shown to converge to $U^\eps$ as $\delta\to 0$ for arbitrary mollifications $\kappa$ in \cite[Theorem 1.6]{OOT}, becomes an exponential equivalence in the low-temperature limit $\eps\to 0$. This may be accomplished by applying Proposition \ref{prop:prototype_wiener_chaos_bound} to the bounds \cite[Lemma 4.4]{OhThomann}, and using Lemma \ref{lem:path-global-reconstruction}.  Related global dynamics for hyperbolic stochastic equations are developed in \cite{GubinelliKochOhTolomeo,oh2021two}.\\ \paragraph*{\em Large Deviations for Dispersive Equations with Random Initial Data}
Motivated by the modelling of rogue waves \cite{DematteisGrafkeVandenEijnden}, a body of recent work has investigated the large deviations of dispersive equations where, as in the current setting, the only source of randomness is the initial data.   We refer to \cite{GarridoGrandeKurianskiStaffilani,LiangWang2025,Grande2025} for {rigorous} results in the cubic Schr\"odinger equation in the weakly nonlinear regime. The limits considered in Theorems \ref{thm:main} - \ref{thm: bad LDP} are of a different form, as the nonlinearity is kept fixed and the temperature $\eps\to 0$. \\ \paragraph*{\em Low-Temperature Gibbs Measures and LDP of Singular SPDE}
{Static large deviations for singular Gibbs measures have been obtained for the $\Phi^4_2$ measure in infinite volume in \cite{barashkov2023variational}, for the $\Phi^4_3$ measure in \cite{KloseMayorcas}, and for the Yang--Mills measure in \cite{levy2006large}; we refer also to \cite{gess2024low} for a sharper low-temperature expansion of the Euclidean $\Phi^4_2$ measure.} The LDP of a Gross-Pitaevskii Gibbs measure was further derived in \cite{PackerSeongSosoe}. For invariant measures of stochastic reaction--diffusion systems, large deviations were established in \cite{cerrai2005large}, and an LDP for the stationary measure of wave equations driven by smooth white noise was proven by Martirosyan \cite{Martirosyan2017}. \\  Large deviations for small-noise reaction-diffusion equations were obtained in \cite{CerraiRoeckner2004}. Hairer-Weber \cite{HairerWeber} proved a large deviation principle for the renormalised dynamical $\Phi^4_d$ equation, $d\in\{2,3\}$, both directly for the renormalised equation and in the joint limit where noise intensity and correlation length vanish simultaneously. Large deviations in the latter scaling limit were also obtained in \cite{CerraiDebusschePhi,CerraiDebusscheNS}. \\The proofs of Theorems \ref{thm:main} - \ref{thm:path-fine-structure} follow particularly the strategy of \cite[Sections~3--4]{HairerWeber}, translated into the dispersive setting. As in \cite{HairerWeber}, the solution theory of \eqref{eq:wick-nlw} factors discontinuous solution map $X^\eps\to U^\eps$ into the canonical lift of Wick powers $X^\eps \mapsto \mathfrak Z^\eps$, followed by a continuous solution map $\Gamma_T$, using a dispersive analogue of the Da Prato-Debussche trick \cite{DaPratoDebussche2003}, see \cite[Remark 1.3]{OhThomann}, \cite{Bourgain1994}. It is therefore sufficient to prove a large deviation principle for the enhanced data $\mathfrak Z^\eps$, and Theorem \ref{thm:main} follows by the contraction principle.   \\ \paragraph*{\em Renormalisation Shifts and Anomalous Limits}
The same framework is also used to prove Theorem~\ref{thm: bad LDP}. Enlarging the probability space if necessary, we construct a finite shift of the initial data in high-frequency modes, vanishing in the limit $\eps\to 0$ in $\cH^s$ with superexponentially high probability. The shifts are engineered to create a finite, non-random shift of the enhanced data, which manifests as a finite shift of the mass. The mechanism is analogous to high-frequency perturbations which generate finite renormalisation corrections in singular SPDEs \cite[Section~2.4]{hairer2013solving}, \cite[Equations (1.7),(1.8)]{hairer2022support}, see also the Ansatz before \cite[Equation (1.18)]{hairer2022support}, where a distributionally small oscillatory perturbation changes an effective renormalisation constant. In contrast to those settings, \eqref{eq:wick-nlw} has no driving noise, and the shifts can only be introduced at the level of the initial data. In the large-deviations setting, \cite[Theorem~4.7, Remark~4.9]{HairerWeber} show how such shifts of renormalisation constants may appear in the rate function. \\ \paragraph*{\em Discontinuity, Norm Inflation and Triviality}
As discussed below Theorem \ref{thm:main}, the contrast between Theorems \ref{thm:main} - \ref{thm: bad LDP} arises as a result of the discontinuity of the data-to-solution map $X^\eps\to U^\eps$. In Section \ref{sec:instability}, we give a result Theorem~\ref{prop:instability} on this instability at fixed temperature, which follows as a result of the construction in Theorem \ref{thm: bad LDP}.  The discontinuity is weaker than the norm blow-up results in  \cite[Theorem~1.11]{OOT}, \cite{xia2021generic,oh2026probabilistic}, and, for deterministic low-regularity NLS, \cite{oh2017remark,guo2018non}.  The second part of Theorem \ref{prop:instability} further relates to {\em triviality} phenomena. For singular stochastic PDE, these have been studied in \cite{albeverio1996trivial,albeverio2001two,hairer2012triviality,oh2020remark,oh2021two}, and we refer to \cite[Section~6]{OhPocovnicuTzvetkov} and \cite[Section~1.3]{OOT} for a discussion of this phenomenon for dispersive equations.

\section{Preliminaries}
\subsection{Preliminaries on NLW}  For a temperature $\eps>0$, let $\gamma^\eps$ be the centred Gaussian measure induced by
\begin{equation}\label{eq:Gaussian-data}
 q_0^\eps=\sqrt\eps\sum_{n\in\Z^2}\frac{g_{0,n}}{\langle n\rangle}e^{in\cdot x},
 \qquad
 q_1^\eps=\sqrt\eps\sum_{n\in\Z^2}g_{1,n}e^{in\cdot x},
\end{equation}
where $\{g_{0,n},g_{1,n}\}_{n\in\mathbb Z^2}$ are standard complex Gaussian variables subject to the reality condition $g_{j,-n}=\overline{g_{j,n}}$; equivalently, after choosing one representative from each pair $\{n,-n\}$, the corresponding variables are independent and
{
\[
 \mathbb E[g_{j,n}\overline{g_{j,n'}}]=\delta_{n,n'},\qquad
 \mathbb E[g_{j,n}g_{j,n'}]=\delta_{n,-n'}.
\]}
The angle bracket is defined as $\langle n\rangle:=(1+|n|^2)^{1/2}$, and $\langle \nabla\rangle$ is the pseudodifferential operator $(1-\Delta)^{1/2}$ with symbol $\langle n\rangle$. For the projection $P_N$ onto modes of frequency ${|n|\le N}$, put
$$
 \sigma_{\eps,N}=\eps\sum_{\abs n\leq N}\langle n\rangle^{-2},
 \qquad H_4(x;\sigma)=x^4-6\sigma x^2+3\sigma^2.
$$
The low-temperature defocusing Gibbs measure is the limit of
\begin{equation}\label{eq:Gibbs-trunc}
 \dd\mu^\eps_N(u,v)
 =Z_{\eps,N}^{-1}
 \exp\left(-\frac1{4\eps}\int_{\T^2}H_4(P_Nu;\sigma_{\eps,N})\,\dd x\right)
 \dd\gamma^\eps(u,v).
\end{equation}
 For $\gamma^\eps$-almost every, or equivalently $\mu^\eps$-almost every, initial data $X^\eps$, there exists a globally defined, canonical solution ${U^\eps}$ to \eqref{eq:wick-nlw}.  
 
 As well as the regularity index $s$ defining the state space $\cH^s$, we fix $\frac12<s_0<1$, as well as {$0<\kappa=\kappa(s_0)<1$} given by \cite[Lemma 4.4]{OOT}, and put $p=2/\kappa$. For an interval $I\subset\mathbb R$, set
\begin{equation} \label{eq:enhanced_space}
 \mathcal Y(I)=\prod_{\ell=1}^3L^p(I;W^{-\kappa,p}(\mathbb T^2)),
 \qquad
 \scrE(I)=C(I;\mathcal H^s)\times\mathcal Y(I).
\end{equation}
{For any interval $I$ and Sobolev exponent $r$, we also write $\cX_I^r=C(I;\cH^r)$, and write a subscript $_t$ to shorten notation when the interval is $I=[-t,t]$. We identify a scalar wave $v$ with the phase-space path $(v,\partial_tv)$.}
\\ \\ For an initial datum $x$, valued in any $\cH^k$, let $z(t):=S(t)x$ be the linear wave solution $(\partial_t^2+1-\Delta)z=0$ started from $x$. The wave propagator $S(t)$ has the explicit form \begin{equation}\label{eq:propagator_0} S(t)(f,g)=\cos(t\langle \nabla \rangle)f+\frac{\sin(t\langle \nabla\rangle)}{\langle \nabla \rangle}g. \end{equation} Throughout, we will write $z^\eps$ for the linear wave solution started at $X^\eps$.\\\\ In Section \ref{sec:instability}, we will write $S_a$ for the same propagator, when the mass 1 is replaced by $1+a$, $a\ge 0$. The propagator $S$ above thus coincides with $S_0$.
\subsection{Preliminaries on LDP} We use the contraction principle in the following form.
\begin{lemma}[Effective Contraction Principle, {\cite[Lemma~3.3]{HairerWeber}}]
\label{lem:path-effective-contraction}
Let $E,F$ be Polish spaces. Suppose $Y^\eps$ satisfies a good LDP on $E$ with rate $I$, and let $G:E\to F$ be Borel measurable. Suppose that there is an open set $D\subset E$ containing $\{I<\infty\}$ such that $G|_D$ is continuous. Then $G(Y^\eps)$ satisfies a good LDP on $F$ with rate
\begin{equation}\label{eq:path-effective-contraction-rate}
 J(y)=\inf\{I(x):G(x)=y\}.
\end{equation}
\end{lemma}
The following estimate allows us to derive superexponential bounds under $\mu^\eps$ from equivalent superexponential bounds under $\gamma^\eps$. 
\begin{lemma}[Gaussian-to-Gibbs transfer]\label{lem:change-measure}
For every $q>1$,
\begin{equation}\label{eq:density-Lq}
 \sup_{0<\eps\leq1} \eps \log
 \norm{\frac{\dd\mu^\eps}{\dd\gamma^\eps}}_{L^q(\gamma^\eps)}<\infty.
\end{equation}
Consequently, if events $A_{\eps,\delta}$ satisfy
\[
 \lim_{\delta\downarrow0}\limsup_{\eps\downarrow0}
 \eps\log\gamma^\eps(A_{\eps,\delta})=-\infty,
\]
then the same statement holds with $\gamma^\eps$ replaced by $\mu^\eps$.  The analogous assertion holds without the $\delta$-limit.
\end{lemma}

\begin{proof}
Let $D_\eps(f,g)=(\sqrt\eps f,\sqrt\eps g)$ and let
\[
 \mathcal R(f)=\frac14\int_{\T^2}:f^4:\,\dd x
\]
be the unit-temperature Wick interaction.  The Hermite scaling identity gives
\[
 \gamma^\eps=(D_\eps)_\#\gamma^1,
 \qquad
 \frac{\dd\mu^\eps}{\dd\gamma^\eps}(D_\eps X)=Z_\eps^{-1}e^{-\eps^{-1}\mathcal R(D_\eps(X))}
 =Z_\eps^{-1}e^{-\eps\mathcal R(X)}.
\]
Since $\E_{\gamma^1}\mathcal R=0$, Jensen's inequality gives $Z_\eps\geq1$.  Moreover, for $0<\eps\leq1$,
\[
 e^{-q\eps\mathcal R}\leq 1+e^{-q\mathcal R}.
\]
Proposition~1.1 of \cite{OhThomann} states that
$e^{-\mathcal R}\in L^r(\gamma^1)$ for every finite $r$, and hence
\[
 \norm{\frac{\dd\mu^\eps}{\dd\gamma^\eps}}_{L^q(\gamma^\eps)}^q
 \leq \E_{\gamma^1}[1+e^{-q\mathcal R}]<\infty,
\]
uniformly in $\eps$.  Letting $q'=q/(q-1)$ be the conjugate index, H\"older's inequality shows that, for any set $A$, $
 \mu^\eps(A)
 \leq C_q\gamma^\eps(A)^{1/q'},
$ which proves both asserted consequences.
\end{proof} 

Since we will use it repeatedly, we record the following tail bound for Wiener chaoses \cite[Theorem~4.1]{Borel84}, see also \cite[Equation~(3.5), Lemma 3.9]{HairerWeber}. \begin{proposition}\label{prop:prototype_wiener_chaos_bound}
	Let $(B,H, \mu)$ be an abstract Wiener space, let $E$ be a real, separable Banach space, and let $\Psi\in \oplus_{k\le \ell} H^k(E)$ {be an} element of the $E$-valued $\ell^{\rm th}$ (inhomogeneous) Wiener chaos. Then \[
 \mu\big(\xi: \norm{\Psi(\xi)}_E>r\big)
 \leq C_\ell\exp\left[-c_\ell
 \left(\frac r{\norm{\Psi}_{L^2(\Omega,E)}}\right)^{2/\ell}\right].
\]

\end{proposition}

We also use the following consequence of Varadhan's Lemma and Bryc's Theorem \cite[Theorems 4.3.1, 4.4.2]{DemboZeitouni}. \begin{proposition}[Transfer of LDP]\label{prop: transfer} Let $E$ be a metric space, and suppose, under a probability measure $\mathbb{P}$, that random variables $(X^\eps, Y^\eps)$ satisfy a large deviation principle in $E\times\mathbb{R}$ at speed $\eps^{-1}$ with good rate function $I(x,y)$.  Suppose that, for some $q>1$, \begin{equation}
	\label{eq: transfer hypothesis} \limsup_{\eps\to 0} \eps \log \mathbb{E}\left[\exp\left({-\eps^{-1} q Y^\eps}\right)\right]<\infty
\end{equation} and define probability measures $\mathbb{Q}^\eps$ by $\frac{d\mathbb{Q}^\eps}{d\mathbb{P}}\propto \exp(-\eps^{-1}Y^\eps)$. Then, under $\mathbb{Q}^\eps$, the random variables $X^\eps$ satisfy a large deviation principle with good rate function \begin{equation}
	\label{eq: transfer rate} I(x):=\inf_y \{I(x,y)+y\}-\inf_{x',y}\{I(x',y)+y\}.
\end{equation}
	
\end{proposition}

\section{Enhanced data and reconstruction}\label{sec:enhanced}

The first step towards Theorems \ref{thm:main} - \ref{thm: bad LDP} is to prove a large-deviation principle for enhanced data. Recalling the definition of $z^\eps=S X^\eps$, and on the canonical full-measure set on which the Wick powers exist, define
\begin{equation}\label{eq: enhanced data}
 \mathfrak Z^\eps
 =\big((z^\eps,\partial_tz^\eps),z^\eps,:(z^\eps)^2:_\eps,:(z^\eps)^3:_\eps\big).
\end{equation}
We now prove a large deviation principle for $\mathfrak Z^\eps$ in the enhanced data space $\mathscr E([-T,T])$. In the sequel we omit the subscript $\eps$ from the Wick powers when the covariance defining the renormalisation is clear from context. In order to specify the rate function, define $\mathfrak L_0:H^1\times L^2\to\scrE([-T,T])$ by
\begin{equation}\label{eq: define mathfrak L}
 \mathfrak L_0(h)=\big((Sh,\partial_tSh),Sh,(Sh)^2,(Sh)^3\big).
\end{equation}
Since $Sh(t)\in H^1$ for every time, the pointwise powers are well-defined and $\mathfrak L_0(h)\in\scrE([-T,T])$.

\begin{proposition}[LDP of Enhanced Data]\label{prop:enhanced-gaussian}
The laws of $\mathfrak Z^\eps$ satisfy an LDP on $\scrE([-T,T])$, with speed $\eps^{-1}$ and good rate function
\begin{equation}\label{eq:enhanced-rate}
 \mathbf I_0(\Xi)
 =\begin{cases}
 \cE(h),&\Xi=\mathfrak L_0(h),\quad h\in H^1\times L^2,\\
 +\infty,&\text{otherwise.}
 \end{cases}
\end{equation}
\end{proposition}

\begin{proof}
We divide into steps. First, we prove the corresponding statement for data from the low-temperature Gaussian free field $\gamma^\eps$, and then transfer it to the Gibbs measure $\mu^\eps$.
\step{1. LDP for GFF Enhanced Data}
Consider initial data $\bar X^\eps=(\bar Q_0^\eps,\bar Q_1^\eps)\sim\gamma^\eps$. {Let $\bar z^\eps=S\bar X^\eps$ and denote by $\bar{\mathfrak Z}^\eps$ the enhancement defined by \eqref{eq: enhanced data} with $z^\eps$ replaced by $\bar z^\eps$.} The phase-space linear wave $(\bar z^1,\partial_t\bar z^1)$ belongs to the first homogeneous Wiener chaos of $\bar X^1$, while $:(\bar z^1)^\ell:$ belongs to the $\ell^{\rm th}$ homogeneous Wiener chaos for $\ell=2,3$, and
\[
 \bar Y^1:=\frac14\int_{\mathbb T^2}:(\bar Q_0^1)^4:\,dx
\]
belongs to the fourth homogeneous Wiener chaos. The estimates in \cite[Proposition 4.2]{OhThomann} {show} that, for each $\ell=2,3$, $:(\bar{z}^1)^\ell: \in {L^2(\gamma^\eps, \mathcal Y([-T,T]))}$.
 {For every $\eps>0$, set
\[
 \bar Y^\eps:=\frac14\int_{\mathbb T^2}:(\bar Q_0^\eps)^4:_\eps\,dx
\]
which is the exponent in the density $d\mu^\eps/d\gamma^\eps\propto e^{-\bar Y^\eps/\eps}$. }The scalings
\[
\begin{aligned}
 &((\bar z^\eps,\partial_t\bar z^\eps),(: (\bar z^\eps)^\ell:)_{1\le\ell\le3},\bar Y^\eps)\\
 &\qquad =_{\rm Law}
 (\eps^{1/2}(\bar z^1,\partial_t\bar z^1),
 (\eps^{\ell/2}:(\bar z^1)^\ell:)_{1\le\ell\le3},\eps^2\bar Y^1).
\end{aligned}
\]
put the tuple exactly in the setting of \cite[Theorem~3.5]{HairerWeber}. Evaluation of the homogeneous chaoses on elements $h=(f,g)\in H^1\times L^2$ of the Cameron-Martin space of $\gamma^1$ gives
\[
 \left(\mathfrak L_0(h),\frac14\int_{\mathbb T^2}f^4(x)\,dx\right).
\]
It follows that $(\bar{\mathfrak Z}^\eps,\bar Y^\eps)$ satisfies an LDP in $\scrE([-T,T])\times\mathbb R$ with good rate function
\begin{equation}\label{eq: enhanced rate GFF}
\begin{aligned}
 \mathbf I_{\rm G}(\Xi,y)
 =\inf\Big\{&\frac12\|h\|_{H^1\times L^2}^2:\
 \Xi=\mathfrak L_0(h),
 &y=\frac14\int_{\mathbb T^2}f^4(x)\,dx,\quad h=(f,g)\Big\}.
\end{aligned}
\end{equation}
\step{2. Transfer to the Gibbs Measure}
We have
\[
 \frac{d\mu^\eps}{d\gamma^\eps}\propto
 \exp(-\eps^{-1}\bar Y^\eps).
\]
The hypothesis \eqref{eq: transfer hypothesis} is verified by Lemma~\ref{lem:change-measure}, and we may therefore apply Proposition~\ref{prop: transfer} to \eqref{eq: enhanced rate GFF}. To deduce the rate function, if $\Xi=\mathfrak L_0(h)$, $h=(f,g)$, then

\[
 \inf_y\{\mathbf I_{\rm G}(\Xi,y)+y\}
 =\frac12\|h\|_{H^1\times L^2}^2+\frac14\int_{\mathbb T^2}f^4\,dx
 =\cE(h)
\] whereas, if $\Xi \not \in \mathfrak{L}_0(H^1\times L^2)$, the corresponding set is empty and the infimum is infinite. The normalising constant in \eqref{eq: transfer rate} is $ \inf_{h\in H^1\times L^2}\cE(h)=0$, producing the claimed rate function \eqref{eq:enhanced-rate}.
\end{proof}

For $\Xi=(\mathbf z,f_1,f_2,f_3)\in\scrE([-T,T])$, consider the remainder equation
\begin{align}
 \partial_t^2v+(1-\Delta)v+v^3+3f_1v^2+3f_2v+f_3&=0,\label{eq:path-remainder}\\
 (v,\partial_tv)|_{t=0}&=(0,0).\label{eq: path-remainder-initial-cond}
\end{align}
For $s_0\in(1/2,1)$ and $\Xi\in\scrE([-T,T])$, a solution exists locally by \cite[Section 3]{OhThomann}, \cite[Lemma~4.4]{OOT}, and is unique as long as it is defined. As throughout, a scalar solution is identified with its phase-space path. For the Gibbs enhanced data $\mathfrak Z^\eps$, the solution is global and the canonical solution of \eqref{eq:wick-nlw} is $U^\eps:=(z^\eps+v,\partial_tz^\eps+\partial_tv)$.

Before giving the statement of the continuous dependence $\mathfrak Z^\eps \to U^\eps$, it is useful to introduce further objects which will be used in Theorem \ref{thm:path-fine-structure} and hence give a single, unified statement. For $m\ge0$, define the map $\mathcal T_m: \scrE([-T,T])\to \scrE([-T,T])$ by
\begin{equation}\label{eq:path-T-kappa-first}
 \mathcal T_m(\mathbf z,f_1,f_2,f_3)
 =\left(\mathbf z,f_1,f_2+\frac m3,f_3+mf_1\right),
\end{equation}
and the shifted version of \eqref{eq: define mathfrak L} given by
\begin{equation}\label{eq: define mathfrak Lm}
 \mathfrak L_m(h)=\mathcal T_m(\mathfrak L_0(h)).
\end{equation}
Define
\begin{equation}\label{eq:enhanced-rate-m}
 \mathbf I_m(\Xi)=
 \begin{cases}
 \cE(h),&\Xi=\mathfrak L_m(h),\quad h\in H^1\times L^2,\\
 +\infty,&\text{otherwise.}
 \end{cases}
\end{equation}
and let $\Phi_T^m$ denote the {energy} solution map for
\begin{equation}\label{eq:classical-nlw-shifted-mass}
 \partial_t^2u+(1+m-\Delta)u+u^3=0,
 \qquad (u,\partial_tu)|_{t=0}=(f,g)\in H^1\times L^2.
\end{equation}
Let $\mathscr D_T\subset\scrE([-T,T])$ be the set of $\Xi$ for which the solution to \eqref{eq:path-remainder}--\eqref{eq: path-remainder-initial-cond} exists on all of $[-T,T]$, and, for $\Xi\in\mathscr D_T$, define
\begin{equation}\label{eq:path-reconstruction}
 \Gamma_T(\Xi)={\mathbf z+(v,\partial_tv)}\in\mathcal X_T^s.
\end{equation}
The result we will use is the following.
\begin{lemma}[Global reconstruction on the effective domain]\label{lem:path-global-reconstruction}
The set $\mathscr D_T$ is open in $\scrE([-T,T])$, and $\Gamma_T$ is continuous on $\mathscr D_T$. Moreover, for every $m\ge0$,
\begin{equation}\label{eq:path-effective-domain-in-D}
 \{\mathbf I_m<\infty\}\subset\mathscr D_T,
\end{equation}
and, for $h\in H^1\times L^2$,
\begin{equation}\label{eq:path-reconstruction-skeleton}
 \Gamma_T(\mathfrak L_m(h))=\Phi_T^m(h).
\end{equation}
\end{lemma}

\begin{proof}
We divide into steps.
\step{1. Openness of $\mathscr D_T$ \& Continuity of $\Gamma_T$}
Given $\Xi\in\mathscr D_T$, let $v$ be the global solution to \eqref{eq:path-remainder}, and pick
\[
 R>\|v\|_{C([-T,T];H^{s_0})}+\|\partial_tv\|_{C([-T,T];H^{s_0-1})}.
\]
Following the estimates \cite[Theorem~4.1]{GubinelliKochOhTolomeo}, \cite[Equations~(4.19)--(4.22)]{OOT} in the existence proof, there exist $\eta=\eta(R)>0$ and $\tau=\tau(R)>0$ such that, on any interval $I=[t_0,t_1]$ of length $|I|\le\tau$, for $\Xi'=(\mathbf z',\mathbf f')\in\scrE(I)$ with $\|\mathbf f'\|_{\mathcal Y(I)}<\eta$, and for phase-space data $u\in\mathcal H^{s_0}$ with $\|u\|_{\mathcal H^{s_0}}\le R$, the fixed-point map
\begin{equation}\label{eq: fixed point map}
\begin{split}
 \Psi_{I,u,\Xi'}(w)(t)
 &=S(t-t_0)u\\
 &\quad-\int_{t_0}^t\frac{\sin((t-t')\langle\nabla\rangle)}{\langle\nabla\rangle}
 (w^3+3f'_1w^2+3f'_2w+f'_3)(t')\,dt'
\end{split}
\end{equation}
defines a contraction on the ball $B_{I,R}:=\{w:\|w\|_{\cX_I^{s_0}}\le R+1\}$, which contains $v|_I$. The map $(u,\Xi')\mapsto\Psi_{I,u,\Xi'}$ is continuous from $\mathcal H^{s_0}\times\mathcal Y(I)$ into the uniform topology of maps $B_{I,R}\to B_{I,R}$. Since the contractions have a common contraction constant strictly less than one after decreasing $\tau$ and $\eta$ if necessary, the fixed point depends continuously on $(u,\Xi')$. Consequently, the solution map to \eqref{eq:path-remainder} is continuous from $$\{\Xi'\in {\scrE}(I), u\in \cH^{s_0}: \|\mathbf{f}'\|_{\mathcal{Y}(I)}\le \eta, \|u\|_{\cH^{s_0}}\le R\} $$ into the topology of $\cX^{s_0}_I$.
For the given $\Xi$, absolute continuity of the maps $t\mapsto\|\mathbf f\|_{\mathcal Y([0,\pm t])}^p$ gives a finite partition
\[
 [-T,T]=\bigcup_{j=0}^{J-1}[t_j,t_{j+1}],\qquad t_k=0,
\]
with $|t_{j+1}-t_j|\le\tau$ and $\|\mathbf f\|_{\mathcal Y([t_j,t_{j+1}])}<\eta/2$ for every $j$. By continuity of restriction in $\scrE([-T,T])$, there is an open neighbourhood $\mathscr U_J$ of $\Xi$ such that every $\Xi'\in\mathscr U_J$ satisfies $\|\mathbf f'\|_{\mathcal Y([t_j,t_{j+1}])}<\eta$ on every interval of the partition.

Fix {$\delta>0$, and choose }${0<\delta_J<\delta}$ with
\[
 \|v\|_{C([-T,T];H^{s_0})}+\|\partial_tv\|_{C([-T,T];H^{s_0-1})}+\delta_J<R.
\]
By local continuous dependence on $[t_{J-1},T]$, there exist an open neighbourhood $\mathscr V_J$ of $\Xi|_{[t_{J-1},T]}$ in $\scrE([t_{J-1},T])$ and $\delta_{J-1}\in(0,\delta_J)$ such that, whenever the phase-space restart datum $u$ satisfies
\[
 \|u-(v(t_{J-1}),\partial_tv(t_{J-1}))\|_{\mathcal H^{s_0}}<\delta_{J-1}
\]
and $\Xi'|_{[t_{J-1},T]}\in\mathscr V_J$, the corresponding solution exists on $[t_{J-1},T]$ and differs from $v$ there by less than $\delta_J$ in $\cX_{[t_{J-1},T]}^{s_0}$. Set
\[
 \mathscr U_{J-1}=\mathscr U_J\cap\{\Xi':\Xi'|_{[t_{J-1},T]}\in\mathscr V_J\}.
\]
Iterating backwards over the finitely many intervals gives open neighbourhoods
\[
 \mathscr U_J\supset\mathscr U_{J-1}\supset\cdots\supset\mathscr U_k\ni\Xi
\]
and numbers $\delta_J>\delta_{J-1}>\cdots>\delta_k>0$ such that every $\Xi'\in\mathscr U_k$ produces a solution $v'$ on $[0,T]$ from the prescribed zero remainder data, satisfying $\|v'-v\|_{{\cX_{[0,T]}^{s_0}}}<\delta$. Repeating the same construction forward from $-T$ to $0$ produces another open neighbourhood $\mathscr U_k'$ with the analogous property on $[-T,0]$. Consequently
\[
 \Xi\in\mathscr U_k\cap\mathscr U_k'\subset\mathscr D_T
\]
 and the solution map $\Xi'\mapsto v$ is continuous at $\Xi$. Since the linear component ${\mathbf z}$ is immediately continuous on $\mathscr E([-T,T])$, the continuity of $\Gamma_T$ follows.

\step{2. Global Existence for Finite Rate Data and Reconstruction Identity}
Let $h=(f,g)\in H^1\times L^2$ and $w=Sh$. Taking $\Xi=\mathfrak L_m(h)$, equation \eqref{eq:path-remainder} becomes
\[
 (\partial_t^2+1-\Delta)v+v^3+3wv^2+3(w^2+\tfrac13m)v+w^3+mw=0.
\]
Since $(\partial_t^2+1-\Delta)w=0$, $u=w+v$ solves \eqref{eq:classical-nlw-shifted-mass}. The latter is defocusing and energy-subcritical, and conservation of
\[
 \cE_m(u,\partial_tu)=\cE(u,\partial_tu)+\frac m2\|u\|_{L^2}^2
\]
gives a global solution. Hence $\mathfrak L_m(h)\in\mathscr D_T$, proving \eqref{eq:path-effective-domain-in-D}; \eqref{eq:path-reconstruction-skeleton} follows by uniqueness.
\end{proof}

{It is convenient to define the Borel extension
\begin{equation}\label{eq:extended-reconstruction}
 \widehat\Gamma_T(\Xi)=
 \begin{cases}
 \Gamma_T(\Xi),&\Xi\in\mathscr D_T,\\
 0,&\Xi\notin\mathscr D_T.
 \end{cases}
\end{equation}
Since $\mathscr D_T$ is open and $\Gamma_T$ is continuous there, $\widehat\Gamma_T$ is Borel measurable and its restriction to the open set $\mathscr D_T$ is continuous.}

\section{Proof of Theorem~\ref{thm:main}}\label{sec:pf_main}

{
\begin{proof}[Proof of Theorem~\ref{thm:main}]
To obtain the static LDP for $X^\eps$, the evaluation map
\begin{equation}\label{eq:eval_map}
 \pi_0:\scrE([-T,T])\to\mathcal H^s,\qquad
 \pi_0(\mathbf z,f_1,f_2,f_3)={\mathbf z(0)},
\end{equation}
satisfies $\pi_0(\mathfrak Z^\eps)=X^\eps$ and $\pi_0(\mathfrak L_0(h))=h$. The contraction principle applied to Proposition~\ref{prop:enhanced-gaussian} therefore produces the claimed LDP with rate $\cE$. For the trajectory LDP of the canonical solution $U^\eps$, the construction of the canonical solution is precisely, in our notation,
\[
 U^\eps=\widehat\Gamma_T(\mathfrak Z^\eps)
\]
for $\mu^\eps$-almost every initial datum. By Lemma~\ref{lem:path-global-reconstruction}, the open set $\mathscr D_T$ contains $\{\mathbf I_0<\infty\}$ and $\widehat\Gamma_T|_{\mathscr D_T}=\Gamma_T$ is continuous. The contraction principle Lemma~\ref{lem:path-effective-contraction} therefore applies to the rate function $\mathbf{I}_0$ found in Proposition~\ref{prop:enhanced-gaussian}, and shows that $U^\eps$ satisfies an LDP with good rate function
\[
 \widetilde J(u)=\inf\{\mathbf I_0(\Xi):\widehat\Gamma_T(\Xi)=u\}
\]
which is readily seen to coincide with the claimed rate function \eqref{eq:dynamic-rate} using the identity \eqref{eq:path-reconstruction-skeleton}.\end{proof}
}

\section{Proof of Theorem \ref{thm: bad LDP}}\label{sec: pathological}

We fix, for the remainder of this section, a mass defect $m\in(0,\infty)$.
Let $\{\Lambda_N\}_{N\ge1}$ be a family of finite subsets of $\Z^2$ satisfying
\begin{equation}\label{eq:path-Lambda-assumptions}
 \Lambda_N=-\Lambda_N,\qquad
 \inf_{n\in\Lambda_N}|n|\ge N,\qquad
 M_N:=|\Lambda_N|\gtrsim N^\theta
\end{equation}
for some $\theta>0$.  For concreteness, the $\Lambda_N$ may be taken to be the annulus
\[
 \Lambda_N=\{n\in\Z^2:N\le |n|<2N\}
\]
with $\theta=2$. We take $N=N_\eps\uparrow\infty$ sufficiently fast that
\begin{equation}\label{eq:path-superfast-N}
 N_\eps^{-1}=o(\eps^k)\qquad\text{for every }k>0.
\end{equation}On an enlargement of the underlying probability space, let
$\{\xi_{0,n},\xi_{1,n}\}_{n\in\Lambda_{N_\eps}}$ be standard complex
Gaussian variables with the same reality convention as in \eqref{eq:Gaussian-data}, and independent of $X^\eps$.  Throughout, we will write $\mathbb P$ and $\mathbb E$ for the probability and expectation of this enlarged space; for objects $X^\eps, U^\eps$ defined already under the original Gibbs measure, these may be equivalently replaced by the Gibbs measure and its expectation.  Put
\begin{equation}\label{eq:path-UV-data}
\begin{aligned}
 Q_0^{\eps,m}
 &=\Big(\frac{m}{3M_{N_\eps}}\Big)^{1/2}
 \sum_{n\in\Lambda_{N_\eps}}\xi_{0,n}e^{in\cdot x},\\
 Q_1^{\eps,m}
 &=\Big(\frac{m}{3M_{N_\eps}}\Big)^{1/2}
 \sum_{n\in\Lambda_{N_\eps}}\langle n\rangle\xi_{1,n}e^{in\cdot x}.
\end{aligned}
\end{equation}
Write $Q^{\eps,m}=(Q_0^{\eps,m},Q_1^{\eps,m})$ and
$q^{\eps,m}=SQ^{\eps,m}$.  The normalisation is chosen so that
\begin{equation}\label{eq:path-q-variance}
 \E[q^{\eps,m}(t,x)^2]=\frac m 3
 \qquad (t,x)\in\R\times\T^2.
\end{equation}
Define
\begin{equation}\label{eq:path-hat-mu}
 \overline X^{\eps,m}=X^\eps+Q^{\eps,m},
 \qquad
 \overline\mu^{\eps,m}=\Law(\overline X^{\eps,m}).
\end{equation}
{First, we note that $\overline\mu^{\eps,m}\ll\gamma^\eps$, since {the finite-dimensional subspace}
\[
 {E_N:=\operatorname{span}\{(e_n,0),(0,e_n):n\in\Lambda_N\}\subset\cH^1}
\]
{supporting $Q^{\eps,m}$ lies in the Cameron-Martin space of $\gamma^\eps$.} Hence, for every deterministic $q \in {E_N}$, writing $\tau_q: \cH^s\to \cH^s$ for the shift $x\mapsto x-q$, we have
\[
 (\tau_q)_\#\mu^\eps\ll(\tau_q)_\#\gamma^\eps\ll\gamma^\eps.
\]
}
In particular, the canonical solution map is defined for $\overline\mu^{\eps,m}$-every initial datum $\overline X^{\eps,m}$, and we let $\overline U^{\eps,m}$ be the canonical solution of
\eqref{eq:wick-nlw}. We will prove the following for any $m>0$.

\begin{theorem}\label{thm:path-fine-structure}
For every $m>0$, the laws
${\{\overline\mu^{\eps,m}\}_{\eps>0}}$ satisfy the same good LDP  on $\cH^s$ as
$\{\mu^\eps\}$, with speed $\eps^{-1}$ and rate $\cE$. However, the canonical solutions $\overline U^{\eps,m}$ satisfy a good LDP on $\cX_T^s$ with rate function given by \begin{equation}\label{eq:bad-dynamic-rate}
	 J_m( u)=
 \begin{cases}
  \cE( u(0)),&u(0)\in H^1\times L^2, \text{ and } u=\Phi_T^m( u(0));\\
  +\infty,&\text{otherwise.}
 \end{cases}
\end{equation}
\end{theorem} Here, we recall that $\Phi_T^m$ is the solution map to the equation \eqref{eq:classical-nlw-shifted-mass} with shifted mass \begin{equation}\label{eq:classical-nlw-shifted-mass-recall}
 \partial_t^2u+(1+m-\Delta)u+u^3=0
\end{equation} for initial data $x=(f,g)\in H^1\times L^2$. In particular, since no $u\neq 0$ satisfies both \eqref{eq:classical-nlw} and \eqref{eq:classical-nlw-shifted-mass-recall}, the rate function $J_m$ given by \eqref{eq:bad-dynamic-rate} differs from \eqref{eq:dynamic-rate}, and the above implies Theorem \ref{thm: bad LDP}.
 \\{The proof follows the same strategy as Theorem \ref{thm:main}, arguing at the level of the enhanced data.  A LDP for the enhanced data $\overline{\mathfrak{Z}}^{\eps, m}$ with rate $\mathbf I_m$ is found in Proposition~\ref{prop:path-anomalous-enhanced}, with technical support given by Proposition \ref{lem:path-UV}. The static LDP then follows by projection to the initial data, while the dynamic LDP follows directly from the contraction principle Lemma \ref{lem:path-effective-contraction} applied to the reconstruction map $\widehat\Gamma_T$.}

\subsection{Main Estimates} Throughout this subsection, we repeatedly write $N=N_\eps$ and $M=M_N$ in order to simplify notation.
\begin{lemma}\label{lem:expequiv} For every $a>0$,
\begin{equation}\label{eq:path-Q-se}
 \limsup_{\eps\downarrow0}\eps\log
 \PP\big(\|Q^{\eps,m}\|_{\cH^s}>a\big)=-\infty.
\end{equation} \end{lemma} \begin{proof}
	We write$$ 
 \|Q^{\eps,m}\|_{\cH^s}^2
 \le C m N^{2s}M^{-1}
 \sum_{n\in\Lambda_N/\{\pm 1\}}(|\xi_{0,n}|^2+|\xi_{1,n}|^2).
$$
where we have enforced independence of the summands by choosing one representative from each pair $\{n,-n\}$. Since each summand has the $\chi^2$-distribution, the cumulant generating function $\psi$ of $|\xi_{0,n}|^2+|\xi_{1,n}|^2$ is finite in a neighbourhood of the origin, independently of $N,m, \eps$. A Chernoff bound yields, for every fixed
$a>0$ and all sufficiently large $N$, and any $\lambda\ge 0$,
\begin{equation}\label{eq:path-Q-tail}\begin{split}
 \PP\big(\|Q^{\eps,m}\|_{\cH^s}>a\big)
& \le \exp\left(-CmM\lambda N^{-2s}{a^2}\right)\mathbb{E}\left[\exp\left(\lambda Cm \sum_{n\in \Lambda_N} (|\xi_{0,n}|^2+|\xi_{1,n}|^2)\right)\right] \\& \leq \exp\left(M\psi(\lambda Cm)-CmM{a^2}\lambda N^{-2s}\right).\end{split}
\end{equation}  Choosing $\lambda>0$ small enough that $\psi(\lambda Cm)<\infty$, we recall
that $s<0$, so, for all sufficiently large $N$, the second term in
\eqref{eq:path-Q-tail} dominates and
\[
 \PP\big(\|Q^{\eps,m}\|_{\cH^s}>a\big)
 \le \exp\big(-c_aMN^{-2s}\big)
 \le \exp\big(-c_aN^{\theta-2s}\big)
\]
which is superexponential in $\eps$ thanks to \eqref{eq:path-superfast-N}.

\end{proof} \begin{proposition}
[Wick Powers of the High-Frequency {Perturbation}]\label{lem:path-UV}

Let $z^\eps=SX^\eps$, where $X^\eps\sim\mu^\eps$ is independent of
$Q^{\eps,m}$.  As $\eps\to 0$, we have the superexponential convergences in $L^p([-T,T], W^{-\kappa,p})$:
\begin{equation} \label{eq:path-q-se}
 q^{\eps,m}\rightarrow 0 \end{equation}
\begin{equation}
 (q^{\eps,m})^2-\frac m 3\rightarrow 0
 \label{eq:path-q2-se}\end{equation} \begin{equation}
 (q^{\eps,m})^3\rightarrow 0
\label{eq:path-q3-se}\end{equation}
\begin{equation} z^\eps q^{\eps,m}\rightarrow 0,
 \qquad
 :(z^\eps)^2: q^{\eps,m}
 \rightarrow 0
\label{eq:path-mixed1-se}\end{equation}
\begin{equation} z^\eps\big((q^{\eps,m})^2-\frac m 3\big)
 \rightarrow 0\label{eq:path-mixed2-se}
\end{equation}
in the sense that, for each of the listed stochastic objects $Y^\eps$, it holds that \begin{equation}\label{eq: claimed supex convergence} \limsup_{\eps\downarrow0}\eps\log\PP(\|Y^\eps\|_{L^p([-T,T],W^{-\kappa,p})}>a)=-\infty. \end{equation}
\end{proposition}
Before giving the proof, we record some tools which will be used in the proof. We begin with the following family of convolution estimates. \begin{lemma}
	\label{lem:convolution} For any $0<{\alpha}<1, b\ge 0$, there exists a constant $C_{\alpha,b}$ such that the following estimates hold. \begin{enumerate}[label=\roman*).] \item {For} all $k\in \Z^2$,  \begin{equation}\label{eq:path-conv-all}
		\sum_{n\in \Z^2} \frac{\log^b(2+\langle n\rangle)}{\langle k+n\rangle^{2\alpha}\langle n\rangle^2} \le C_{b,\alpha}\langle k\rangle^{-2\alpha}\log^{1+b}(2+\langle k\rangle).
	\end{equation} \item For every finite subset $A\subset\Z^2$
of size $|A|=M$, it holds that
\begin{equation}\label{eq:path-finite-set-sum}
 \sup_{\ell\in\Z^2}
\frac1M \sum_{n\in A}
 \frac{\log^b(2+\langle \ell+n\rangle)}
      {\langle \ell+n\rangle^{2\alpha}}
 \le C_{\alpha,b}
 M^{-\alpha}\log^b(2+M).
\end{equation}
\end{enumerate}
\end{lemma}

\begin{proof}
For \eqref{eq:path-conv-all}, we split the sum into the regions
\[
 A_1=\{|n|\le |k|/2\},\qquad A_2=\{|k+n|\le |k|/2\},
\]
and divide the complement by 
\[
A_3=\{|n|\le 2|k|\}\setminus (A_1\cup A_2);\qquad A_4=\{|n|>2|k|\}\setminus (A_1\cup A_2).
\] On $A_1$, $\langle k+n\rangle\asymp\langle k\rangle$, and hence
\[
 \sum_{A_1}\frac{\log^b(2+\langle n\rangle)}{\langle k+n\rangle^{2\alpha}\langle n\rangle^2}
 \lesssim \langle k\rangle^{-2\alpha}\log^{1+b}(2+\langle k\rangle).
\]
On $A_2$, $\langle n\rangle\asymp\langle k\rangle$. Writing $j=k+n$ and using $0<\alpha<1$,
\[
 \sum_{A_2}\frac{\log^b(2+\langle n\rangle)}{\langle k+n\rangle^{2\alpha}\langle n\rangle^2}
 \lesssim \frac{\log^b(2+\langle k\rangle)}{\langle k\rangle^2}
 \sum_{|j|\le |k|/2}\langle j\rangle^{-2\alpha}
 \lesssim \langle k\rangle^{-2\alpha}\log^b(2+\langle k\rangle).
\]
For $A_3$, both $|n|, |n+k|>|k|/2$, so the summand is at most $\langle k\rangle^{-2-2\alpha}\log^b(1+\langle k\rangle)$. Since the set is of size $|A_3|\lesssim |k|^2$, the total contribution at most of the order $\langle k\rangle^{-2\alpha}\log^b(2+\langle k\rangle)$. Finally, on $A_4$, $\langle k+n\rangle\asymp\langle n\rangle$, so a dyadic decomposition gives
\[
 \sum_{A_4}\frac{\log^b(2+\langle n\rangle)}{\langle k+n\rangle^{2\alpha}\langle n\rangle^2}
 \lesssim\sum_{|n|>2|k|}\frac{\log^b(2+\langle n\rangle)}{\langle n\rangle^{2+2\alpha}}
 \lesssim\langle k\rangle^{-2\alpha}\log^b(2+\langle k\rangle).
\]
Combining the four regions proves \eqref{eq:path-conv-all}. \\ \\ For the second claim \eqref{eq:path-finite-set-sum}, let $A,M$ be as given, and fix $\ell\in \Z^2$. We decompose the translate $A+\ell$ into its intersections with dyadic annuli by $$A_{\ell,j}:=\{n\in A+\ell:2^j\le\langle n\rangle<2^{j+1}\}, \qquad 0\le j<\infty.$$  We observe that
$|A_{\ell,j}|\lesssim \min\{M,2^{2j}\}$, and choose $J$ so that
$2^{2J}\asymp M$. The contribution of every $n\in A_{\ell,j}$ is $\asymp 2^{-2j\alpha}(1+j)^b$, and so the sum is bounded, up to constants, by
\begin{equation}\begin{split}
  \sum_{j=0}^\infty |A_{\ell,j}|2^{-2\alpha j}(1+j)^b& \lesssim \sum_{j\le J}2^{2j(1-\alpha)}(1+j)^b
 +\sum_{j>J}M2^{-2\alpha j}(1+j)^b\\ &
 \lesssim M^{1-\alpha}\log^b(2+M),
\end{split}\end{equation}
which proves \eqref{eq:path-finite-set-sum}.
\end{proof}

 The other tool which will be used is a general condition for proving such convergences, which reduces each of the statements of Proposition \ref{lem:path-UV} to estimating the Fourier coefficients. \begin{lemma}\label{lemma:repeatable_chaos}
	Fix $\ell<\infty$, and let $Y=(Y_t)_{-T\le t\le T}\in C([-T,T],\mathscr S'(\T^2))$ be a distribution-valued stochastic process such that, for each ${-T\le t\le T}$, $Y_t$ is a spatially stationary distribution, and belongs to the sum of the Wiener chaoses of order up to $\ell^{\rm th}$ generated by a Gaussian random variable $\mathbb{X}$ on a probability space $(\Omega, \mathcal{F}, \PP)$. Assume that, for some $\kappa>0, A\in (0,\infty)$, \begin{equation}\label{eq:useful_concentration_assumption}
		\sup_{{-T\le t\le T}} \sum_{n\in \Z^2} \langle n\rangle^{-2\kappa} \E|\hat{Y}(t,n)|^2 \le A<\infty.
	\end{equation} Then, for any $2\le p<\infty$, $Y$ admits a version $\widetilde{Y}$ valued in $L^p([-T,T], W^{-\kappa,p})$, and further satisfies the concentration bound, for some {constant $C=C(T,p,\ell)$} and all $r>0$, \begin{equation}
		\label{eq:useful_concentration} \PP\left(\|{\widetilde{Y}}\|_{L^p([-T,T],W^{-\kappa,p})}>r\right)\le C^{-1} \exp\left[-C\left(\frac{r}{\sqrt{A}}\right)^{2/\ell}\right].
	\end{equation}
\end{lemma}

\begin{proof} Let $\rho \in C^\infty_c(\T^2, [0,\infty))$ be a fixed, compactly supported mollifier with unit mass, and let $\rho_\delta, \delta>0$ be the corresponding smooth approximations of the identity. For each $t\in [-T,T]$, for all $\delta>0$ and all $x\in \T^2$, let $Y_\delta(t,x):=Y(t)(\rho_\delta(\cdot-x))$ be the mollification of $Y(t)$ by $\rho_\delta$. Using {Plancherel}'s identity, we estimate \begin{equation}
	\E\left[\|\langle \nabla \rangle^{-\kappa}Y_\delta(t)\|_{L^2}^2\right] = \sum_{n\in \Z^2} \langle n\rangle^{-2\kappa} \E[{|\widehat{Y}(t,n)\widehat{\rho_\delta}(n)|^2}] \le A
\end{equation} where we used the definition of $Y_\delta$ and that $|\widehat{\rho_\delta}|\le 1$. Since the convolution commutes with the pseudodifferential operator, and the stationarity of $Y(t)$ implies the stationarity of $\langle \nabla\rangle^{-\kappa}Y(t)$, we have the bound, uniform in time and space, \begin{equation}
	\E[|(\langle \nabla \rangle^{-\kappa}Y)_\delta(t,x)|^2] \le A.
\end{equation}
	By assumption, each $(\langle \nabla \rangle^{-\kappa}Y)_\delta(t,x)$ belongs to the ${\ell}^{\rm th}$ Wiener chaos associated to $\mathbb X$, and we may apply Proposition \ref{prop:prototype_wiener_chaos_bound} to find, for any $1<p<\infty$,		$
	\E[|(\langle \nabla \rangle^{-\kappa}Y)_\delta(t,x)|^p] \le C_pA^{p/2}$, and integrating in both variables, \begin{equation}
		\E\left[\|Y_\delta\|_{L^p([-T,T],W^{-\kappa,p})}^p\right]=		\E\left[\|\langle \nabla\rangle^{-\kappa}Y_\delta\|_{L^p({[-T,T];L^p})}^p\right] \le C_pT A^{p/2}.
	\end{equation} By lower semicontinuity, we may take $\delta\to 0$ to obtain that $Y\in L^p([-T,T],W^{-\kappa,p})$ almost surely, as claimed. {Since $p\ge2$, the preceding estimate gives $$\|Y\|_{L^2(\PP;L^p([-T,T],W^{-\kappa,p}))}\lesssim_{T,p}\sqrt A.$$ The concentration bound \eqref{eq:useful_concentration} then follows from Proposition~\ref{prop:prototype_wiener_chaos_bound}.}
	
\end{proof} We are now ready to give the
\begin{proof}[Proof of Proposition \ref{lem:path-UV}]  In order to shorten notation, we write $$q:=q^{\eps,m}, \qquad z:=z^\eps, \qquad \mathcal B:=L^p([-T,T],W^{-\kappa,p}).$$ It is convenient, for \eqref{eq:path-mixed1-se}-\eqref{eq:path-mixed2-se}, to work with the case where $X^\eps\sim \gamma^\eps$, which suffices to prove the claim for $X^\eps\sim \mu^\eps$ thanks to Lemma \ref{lem:change-measure}. With this reduction, in the subsequent steps we apply Lemma \ref{lemma:repeatable_chaos} to each object in turn. {Except for $q^3$, which is treated separately in Step~3, the objects to which we apply Lemma~\ref{lemma:repeatable_chaos} below belong to homogeneous Wiener chaoses} generated by the Gaussian random variable $(X^\eps, Q^{\eps,m})\in (\cH^s)^2$. Further, each of the objects listed is spatially stationary, as a result of the spatial stationarity of $z, q$ and their independence. We further observe that it is sufficient to exhibit any $A=A_{\eps,N}$ meeting the assumption \eqref{eq:useful_concentration_assumption} and which decays polynomially in $N$, since the choice \eqref{eq:path-superfast-N} shows that any bound of the form $$ \PP(\|Y^{\eps}\|_{\mathcal{B}} >a) \le C^{-1}\exp\left(-C(a{N^{\theta}})^{2/\ell}\right)$$ with $\theta>0, \ell<\infty$, decays superexponentially fast in $\eps$.\\ \step{1. Convergence of $q^{\eps,m}$} {The initial data $Q^{\eps,m}$ is the componentwise, Gaussian randomisation of the function \[
 p^{\eps,m}:=\sqrt{\frac{m}{3M}}\sum_{n\in\Lambda_N}(1,\langle n\rangle)e^{in\cdot x}\in\mathcal H^{-\kappa}.
\]
By \eqref{eq:path-Lambda-assumptions}, the norm is at most
\[
 \|p^{\eps,m}\|_{\cH^{-\kappa}}^2
 \lesssim_m \frac1M\sum_{n\in\Lambda_N}\langle n\rangle^{-2\kappa}
 \lesssim_m N^{-2\kappa},
\]
and hence $\|p^{\eps,m}\|_{\cH^{-\kappa}}\lesssim_m N^{-\kappa}$.
} Since $\langle\nabla\rangle^{-\kappa}$ commutes with both the wave propagator and the Fourier randomisation, the probabilistic Strichartz estimate \cite[Lemma~2.4]{OOT}, see also \cite{BenyiOh,BurqTzvetkov,CollianderOh}, applied to the randomisation of ${p^{\eps,m}}$ gives, for some constants $c,C>0$ and every $a>0$,
\begin{equation}\label{eq:path-q-tail}\begin{split}
		\mathbb{P}\left(\|q\|_{\mathcal B}>a\right)& = \mathbb{P}\left({\|S(\langle \nabla \rangle^{-\kappa} Q^{\eps,m})\|_{L^p([-T,T],L^p)}}>a\right)\\[1ex] & \le C\exp\left(-c\frac{a^2}{T^{2/p} \|\langle \nabla\rangle^{-\kappa} {p^{\eps,m}}\|_{\cH^{0}}^2 }\right) \\[1ex] & \le C\exp\left(-c\frac{a^2}{T^{2/p} \|{p^{\eps,m}}\|_{\cH^{-\kappa}}^2 }\right).\end{split}\end{equation} Using the previously remarked norm of ${p^{\eps,m}}$, the exponent in the final expression is $\sim {N^{2\kappa}}$, which is $\gg \eps^{-1}$ thanks to \eqref{eq:path-superfast-N}. \step{2. Quadratic Estimate} We start by writing $\hat{\xi}_n(t)$ for the complex Gaussians $\hat{\xi}_n(t):=\xi_{n,0}\cos(\langle n\rangle t)+\xi_{n,1}\sin(\langle n\rangle t)$ for $n\in \Lambda_N$. With this notation, we have \begin{equation}
\begin{split}
 |\widehat{q^2}(k,t)|^2
 &=\left(\frac{m}{3M}\right)^2
 \sum_{n_1,n_2,n_1',n_2'\in\Lambda_N}
 \hat\xi_{n_1}(t)\hat\xi_{n_2}(t)
 \overline{\hat\xi_{n_1'}(t)}\overline{\hat\xi_{n_2'}(t)}\\
 &\hspace{3cm}\times
 \mathbf 1_{\{n_1+n_2=n_1'+n_2'=k\}}.
\end{split}
\end{equation} We estimate the expectation using Isserlis's theorem\footnote{Wick's Theorem.}, recalling that the reality condition produces \begin{equation} \E[\hat{\xi}_n(t)\hat{\xi}_{n'}(t)]=\delta_{n,-n'}; \qquad \E[\hat{\xi}_n(t)\overline{\hat{\xi}_{n'}(t)}]=\delta_{n,n'}. \label{eq:basic_isserlis}\end{equation} 
	
	For $k\in\Z^2$, put
\[
 r_2(k):=\#\big\{(n_1,n_2)\in\Lambda_N^2:n_1+n_2=k\big\}.
\]
The three different possible pairings give \begin{equation}\begin{split} \label{eq:expandquadratic}
		\E |\widehat{q^2}(k,t)|^2 & = 2\left(\frac{m}{3M}\right)^2\sum_{n_1,n_2,n'_1,n'_2}\E[\hat{\xi}_{n_1}(t)\overline{\hat{\xi}_{n'_1}(t)}]\E[{\hat{\xi}_{n_2}(t)}\overline{\hat{\xi}_{n'_2}(t)}] \\ &+ \left(\frac{m}{3M}\right)^2\sum_{n_1,n_2,n'_1,n'_2}\E[\hat{\xi}_{n_1}(t)\hat{\xi}_{n_2}(t)]\E[\overline{\hat{\xi}_{n'_1}(t)}\overline{\hat{\xi}_{n'_2}(t)}]
\end{split}	\end{equation} where the sum runs over the same constrained subset of $\Lambda_N^4$ as before. For every $k\ne0$, the second sum vanishes, while the first pairing gives, by \eqref{eq:basic_isserlis} and the definition of $r_2$,
\[
 \E|\widehat{q^2}(k,t)|^2\lesssim_m M^{-2}r_2(k).
\]
At $k=0$, using $\Lambda_N=-\Lambda_N$,
\[
 \widehat{q^2}(0,t)-\frac m3
 =\frac{m}{3M}\sum_{n\in\Lambda_N}(|\hat\xi_n(t)|^2-1),
\]
and the same bound holds for the centred zero mode.  Thus, uniformly in
$t$,
\begin{equation}\label{eq: quadratic fourier}
 \E\big|\widehat{(q^2-m/3)}(k,t)\big|^2
 \lesssim_m M^{-2}r_2(k),\qquad k\in\Z^2.
\end{equation}
Consequently, rewriting the sum and using \eqref{eq:path-finite-set-sum} twice,
\begin{align*}
 \sum_{k\in\Z^2}\langle k\rangle^{-2\kappa}
 \E\big|\widehat{(q^2-m/3)}(k,t)\big|^2
 &\lesssim_m M^{-2}
   \sum_{n_1,n_2\in\Lambda_N}
   \langle n_1+n_2\rangle^{-2\kappa}\\
 &\lesssim_m M^{-\kappa}
 \lesssim_m N^{-\theta\kappa}.
\end{align*}
Following the discussion at the start of the proof, this proves
\eqref{eq:path-q2-se}. \\  \step{3. Cubic Estimate}
For $H_3$ the third Hermite polynomial, set
\[
 p(t,x):=H_3\big(q(t,x);m/3\big)=q(t,x)^3-mq(t,x).
\]
Since $\E[q(t,x)^2]=m/3$, the field $p(t)$ belongs to the third homogeneous
Wiener chaos. Writing $\mathcal W$ for the Wick product \[ \mathcal{W}(\xi_{n_1},\xi_{n_2},\xi_{n_3}):=\xi_{n_1}\xi_{n_2}\xi_{n_3}
-\E[\xi_{n_1}\xi_{n_2}]\xi_{n_3}
-\E[\xi_{n_1}\xi_{n_3}]\xi_{n_2}
-\E[\xi_{n_2}\xi_{n_3}]\xi_{n_1},\]see, for example \cite[Section~3.1]{JansonGaussianHilbert}, it follows that ${\mathcal W}(\xi_{n_1},\xi_{n_2},\xi_{n_3})$ is the projection of the product ${\xi_{n_1}\xi_{n_2}\xi_{n_3}}$ into the third homogeneous Wiener chaos, and the Fourier transform may be written as \begin{equation} \hat{p}(t,k)=\left(\frac{m}{3M}\right)^{3/2}\sum_{n_1,n_2,n_3\in \Lambda_N}{\mathcal W}(\hat{\xi}_{n_1}{(t)},\hat{\xi}_{n_2}{(t)},\hat{\xi}_{n_3}{(t)})1(k=n_1+n_2+n_3).\end{equation}   Using Isserlis's theorem and \eqref{eq:basic_isserlis}, the Wick products have the orthogonality property \begin{equation}\label{eq:orthogonality_wick} \E[{\mathcal W}(\hat{\xi}_{n_1}{(t)},\hat{\xi}_{n_2}{(t)},\hat{\xi}_{n_3}{(t)})\overline{{\mathcal W}(\hat\xi_{n'_1}{(t)},\hat\xi_{n'_2}{(t)},\hat\xi_{n'_3}{(t)})}]=\sum_{\sigma\in \text{Sym}(3)}\prod_{i=1}^3\delta_{n_i n'_{\sigma(i)}}.\end{equation} Arguing as in \eqref{eq:expandquadratic} - \eqref{eq: quadratic fourier}, we therefore find, uniformly in $t$, the bound $ \E|\widehat p(k,t)|^2\lesssim_m M^{-3}r_3(k)$, where $r_3(k)$ is defined by
\[
 r_3(k):=\#\big\{(n_1,n_2,n_3)\in\Lambda_N^3:
                  n_1+n_2+n_3=k\big\},
\]
Using \eqref{eq:path-finite-set-sum} three times,
\begin{align*}
 \sum_{k\in\Z^2}\langle k\rangle^{-2\kappa}
 \E|\widehat p(k,t)|^2
 &\lesssim_m M^{-3}
   \sum_{n_1,n_2,n_3\in\Lambda_N}
   \langle n_1+n_2+n_3\rangle^{-2\kappa}\\
 &\lesssim_m M^{-\kappa}
 \lesssim_m N^{-\theta\kappa}.
\end{align*}
Lemma~\ref{lemma:repeatable_chaos} therefore shows that $p\to0$
superexponentially in $\mathcal B$.  Since $q\to0$ superexponentially by
Step~1 and $q^3=p+mq$, this proves \eqref{eq:path-q3-se}.\\\step{4. Mixed Terms} We next estimate the terms involving both $z, q$, recalling as at the start of the proof that the analysis has been reduced to the case where $z=SX^\eps, X^\eps\sim \gamma^\eps$ is Gaussian, and independent of $q$. For the product of linear terms, we have \begin{equation}
	\widehat{zq}(k,t)={\sqrt{\frac{m}{3M}}\sum_{n\in \Lambda_N}\widehat{z}(k-n,t)\hat{\xi}_n(t)}
\end{equation} which is formally justified as, for each fixed $\eps$, $q$ is smooth and the product is well-defined. {Using the independence of $z$ and $q$ together with orthogonality of the Fourier coefficients, the summands are orthogonal in $L^2(\mathbb P)$.} In each term, the two coefficients are independent of each other, and $\E|\widehat{z}(n,t)|^2=\eps\langle n\rangle^{-2}$ gives \begin{equation}
	 \E|\widehat{zq}(j,t)|^2=\frac{m\eps}{3M}
   \sum_{n\in\Lambda_N}\frac1{\langle j-n\rangle^2}.
 \label{eq:path-Fourier-zq}
\end{equation} It follows that \begin{equation}
\begin{split}
 \sum_{k\in \Z^2}\langle k\rangle^{-2\kappa} \E[|\widehat{zq}(k,t)|^2]
 &=\frac{m\eps}{3M}\sum_{k\in\mathbb Z^2}\sum_{n\in\Lambda_N}
   \frac1{\langle k\rangle^{2\kappa}\langle k-n\rangle^2}\\
 &=\frac{m\eps}{3M}\sum_{n\in\Lambda_N}\sum_{k\in\mathbb Z^2}
   \frac1{\langle k+n\rangle^{2\kappa}\langle {k}\rangle^2}.
\end{split}
\end{equation} The innermost sum is bounded $\lesssim_m \eps n^{-2\kappa}\log(2+n)$ by Lemma~\ref{lem:convolution}.  Since
$|n|\ge N$ on $\Lambda_N$, and
$r^{-2\kappa}\log(2+r)$ is decreasing for all sufficiently large $r$, we may
{take }$A\lesssim_m \eps N^{-2\kappa}\log(2+N)$.
 \\\\ 
For the other mixed terms, the proof is similar. For $:z^2:q$, the sum decouples using orthogonality of Fourier modes of $q$; using the bound  $${\E|\widehat{:z^2:}(t,n)|^2} \le C{\eps^2}\log(2+\langle n\rangle)\langle n\rangle^{-2}$$ and Lemma \ref{lem:convolution}, we obtain \begin{equation}\begin{split}
	\sum_{k\in \Z^2}\langle{k}\rangle^{-2\kappa}\E |\widehat{:z^2:q}(t,k)|^2 &\lesssim_m \frac{\eps^2}{M}\sum_{n\in \Lambda_N}\sum_{j\in \Z^2}\frac{\log (2+\langle{j}\rangle)}{\langle{j}+n\rangle^{2\kappa}\langle{j}\rangle^2} \\& \lesssim_m \frac{\eps^2}{M}\sum_{n\in \Lambda_N}\langle n\rangle^{-2\kappa}\log^2(2+\langle n\rangle)\\& \lesssim_m \eps^2 N^{-2\kappa}\log^2(N). \end{split}\end{equation} 
For the term $z(q^2-m/3)$ in \eqref{eq:path-mixed2-se}, we use the independence of $z$ and $q$, and orthogonality of the Fourier modes of $z$ to decouple the sum. Using the bound \eqref{eq: quadratic fourier} obtained in Step 2, we find \begin{equation}\begin{split}
	\sum_{k\in \Z^2}\langle{k}\rangle^{-2\kappa}\E[|\widehat{z(q^2-m/3)}(t,k)|^2] &\lesssim_m
 \frac{\eps}{M^2}
 \sum_{n_1,n_2\in\Lambda_N}
 \sum_{{j}\in\Z^2}
 \frac{1}{\langle{j}+n_1+n_2\rangle^{2\kappa}\langle{j}\rangle^2} \\&\qquad\lesssim_m
 \frac{\eps}{M^2}
 \sum_{n_1,n_2\in\Lambda_N}
 \langle n_1+n_2\rangle^{-2\kappa}
 \log(2+\langle n_1+n_2\rangle)\\
 &\qquad\lesssim_m
 \eps N^{-\theta \kappa}\log(2+N).\end{split}\end{equation}  By \eqref{eq:path-Lambda-assumptions}, the last quantity decays polynomially
in $N$, and Lemma~\ref{lemma:repeatable_chaos} may be applied as described above.
\end{proof}

\subsection{LDP for Enhanced Data}
Let $\mathfrak Z^\eps$ be the enhanced data \eqref{eq: enhanced data} associated with $X^\eps$, and let $\overline{\mathfrak Z}^{\eps,m}$ be the enhanced data associated with $\overline X^{\eps,m}$, obtained by replacing $z^\eps$ by $z^\eps+q^{\eps,m}$. We emphasise that the covariance used in the Wick ordering is, for all objects, that of the original Gaussian free field $\gamma^\eps$. In particular, the Wick powers of $z^\eps+q^{\eps,m}$ are Wick ordered with respect to the covariance of the original Gaussian field $\gamma^\eps$, and not with respect to the enlarged covariance after adding $Q^{\eps,m}$.

\begin{proposition}[Enhanced LDP with Perturbation]\label{prop:path-anomalous-enhanced}
The families
$\overline{\mathfrak Z}^{\eps,m}$ and
$\mathcal T_m(\mathfrak Z^\eps)$ are exponentially equivalent on
$\scrE([-T,T])$. Consequently,
$\overline{\mathfrak Z}^{\eps,m}$ satisfies a good LDP with speed
$\eps^{-1}$ and rate $\mathbf I_m$ defined in \eqref{eq:enhanced-rate-m}.
\end{proposition}

\begin{proof}
{Using the Hermite binomial identities}
\begin{align*}
 :(z+q)^2:&=:z^2:+2zq+q^2,\\
 :(z+q)^3:&=:z^3:+3:z^2: q+3zq^2+q^3
\end{align*}
we write \begin{equation}\label{eq:enhanced_difference} \overline{\mathfrak Z}^{\eps,m}-\mathcal T_m(\mathfrak Z^\eps)=\big((q,\partial_tq),(q,2zq+q^2-m/3,3:z^2:q+3z(q^2-m/3)+q^3)\big)\end{equation} and each term converges to zero with superexponential probability by Lemma \ref{lem:expequiv}, and Proposition ~\ref{lem:path-UV}. For the claimed large deviation principle, {applying the contraction principle to} the continuous map $\mathcal T_m:\scrE([-T,T])\to\scrE([-T,T])$ given by \eqref{eq:path-T-kappa-first} and the large deviation principle for $\mathfrak Z^\eps$ in Proposition \ref{prop:enhanced-gaussian} show that $\mathcal T_m(\mathfrak Z^\eps)$ satisfies a large deviation principle with rate $\mathbf I_m$. The exponential equivalence established in the first part of the proposition then shows, by \cite[Theorem~4.2.13]{DemboZeitouni}, that $\overline{\mathfrak Z}^{\eps,m}$ satisfies the same large deviation principle. 

\end{proof}
We finally give the
{
\begin{proof}[Proof of Theorem~\ref{thm:path-fine-structure}]
The evaluation map
\[
 \pi_0:\scrE([-T,T])\to\mathcal H^s,\qquad
 \pi_0(\mathbf z,f_1,f_2,f_3)={\mathbf z(0)}
\]
is continuous by the definition of $\scrE$, and $\overline X^{\eps,m}=\pi_0(\overline{\mathfrak Z}^{\eps,m})$. Applying the contraction principle to Proposition~\ref{prop:path-anomalous-enhanced} gives a good LDP for $\overline X^{\eps,m}$ with rate
\[
 I_{0,m}(h)=\inf\{\mathbf I_m(\Xi):\pi_0(\Xi)=h\}=\cE(h).
\]

For the dynamic statement, the canonical construction gives
\[
 \overline U^{\eps,m}=\widehat\Gamma_T(\overline{\mathfrak Z}^{\eps,m})
\]
for $\overline\mu^{\eps,m}$-almost every initial datum. Lemma~\ref{lem:path-global-reconstruction} gives $\{\mathbf I_m<\infty\}\subset\mathscr D_T$, and $\widehat\Gamma_T$ is continuous on the open set $\mathscr D_T$. Lemma~\ref{lem:path-effective-contraction} and Proposition~\ref{prop:path-anomalous-enhanced} therefore give a good LDP with rate
\[
 \widetilde J_m(u)=\inf\{\mathbf I_m(\Xi):\widehat\Gamma_T(\Xi)=u\}=\inf\{\mathcal{E}(h): \widehat{\Gamma}_T(\mathfrak{L}_m(h))=u\}
\]
which {can be easily seen to coincide with} the rate function claimed in \eqref{eq:bad-dynamic-rate} by the identity \eqref{eq:path-reconstruction-skeleton}.
\end{proof}
}

\section{Instability of the Solution at Fixed Temperature} \label{sec:instability}

We record the following instability theorem for the solution, which follows from the construction given in Section \ref{sec: pathological}{.}
{ \begin{theorem}\label{prop:instability} For $i=1,2$, there exist sequences \begin{equation} (w^{N,i})_{N\ge 1}\subset C^\infty\times C^\infty; \qquad w^{N,i}\to 0 \text{ in }\cap_{s<0}\cH^s\end{equation} with the following properties. For $X$ distributed according to the Gibbs measure $\mu^1$, the shift $X^{N,1}:=X+w^{N,1}$ has $\Law(X^{N,1})\ll \mu^1=\Law(X)$, and $X^{N,1}\to X$ in $\cH^{s}$ for every $s<0$, but the canonical solutions $U^{N,1}$ to \eqref{eq:wick-nlw} started at $X^{N,1}$ converge in $\cX^s_T$ to a limit {distinct from} the canonical solution $U$ started at $X$. \\ \\ Similarly, $X^{N,2}:=X+w^{N,2} \to X$ in $\cH^s$ {for every $s<0$}, and $\Law(X^{N,2})\ll \Law(X)$, but the canonical solution $U^{N,2}$ started at $X^{N,2}$ converges to 0 in the topology of {{$H^{-1-\eta}([-T,T];\cH^s)$} for every $s<0$ and $\eta>0$}, and has, almost surely, no convergent subsequences in the topology of $\cX^s_\delta$ for any $s<0, \delta>0$. \\ \\ In both of the above assertions, enlarging the underlying probability space if necessary, the sequences $w^{N,i}$ may alternatively be chosen to be centred, Gaussian and independent of $X$. \end{theorem}
} In the statement above, we emphasise that the same small shift {$w^{N,i}$} produces the discontinuity for almost every $X$, in contrast to the existence of small, $X$-dependent shifts.

\begin{proof}
We prove first the assertion with $w^{N,i}$ replaced by Gaussian $W^{N,i}$, on the probability space $$(\Omega, \mathcal F,\mathbb P):=(\cap_{s<0}\cH^s,\mathcal{B}(\cap_{s<0}\cH^s),\mu^1)\otimes (\C^{\Z^2},\mathcal{B}(\C^{\Z^2}),\mathcal{N}_\C(0,1)^{\otimes \Z^2/\{\pm 1\}}) $$ where the space supporting $X$ is enlarged by a sequence $\xi_{n,i}, i=0,1, n\in \Z^2$ of complex normal random variables, independent modulo the reality constraint. In the remainder of the proof, we construct the Gaussian $W^{N,i}$ in the first homogeneous Wiener chaos of $\{\xi_{n,i}\}$ and hence independent of $X$, for which the assertions hold almost surely. {The existence of the deterministic $w^{N,i}$ follows by choosing any element of the full-measure set on which all assertions hold.} \\\\ For any fixed $m>0$, let $Q^{N,m}$ be as given in \eqref{eq:path-UV-data}, with now $\eps=1$ fixed and $N$ free, and let $W^{N,1}=Q^{N,m}$, for any $m>0$ fixed. The absolute continuity {was already remarked below} \eqref{eq:path-hat-mu}{ together with $\mu^1\sim\gamma^1$}, and $Q^{N,m}$ are smooth and a linear combination of (finitely many) $\xi_{n,i}$. Letting ${\mathfrak{Z}}, \mathfrak{Z}^{N,m}$ be the enhanced data associated to $X, X+Q^{N,m}$ respectively, the stretched exponential estimates in Lemma \ref{lem:expequiv} and Proposition \ref{lem:path-UV} show that $Q^{N,m}\to 0$ in $\cH^s$ almost surely for every $s<0$, and that $\mathfrak Z^{N,m}\to \mathcal T_m\mathfrak Z$ in $\mathscr E([-T,T])$ almost surely. In Lemma \ref{lem:shifted-enhancement-global} below, we show that $\mathcal T_m\mathfrak Z \in \mathscr D_T$ almost surely, and that the reconstruction $\Gamma_T(\mathcal T_m \mathfrak Z)$ is the global solution to \begin{equation}\label{eq:shifted-old-wick}
    \partial_t^2U_m+(1+m-\Delta)U_m+:U_m^3:=0.
\end{equation} {Here and below in this section, the unlabelled Wick product is taken with respect to the original covariance $\gamma^1$.} By the continuity of $\Gamma_T$ on $\mathscr D_T$, it follows that the canonical solution $U^{N,1}$ starting at $X^{N,1}$ is given by \begin{equation}\label{eq:convergence_fixed_m} U^{N,1}=\Gamma_T(\mathfrak Z^{N,m})\to \Gamma_T(\mathcal T_m\mathfrak Z)=U_m \end{equation} with convergence in the topology of $\cX^s_T$. Comparing \eqref{eq:shifted-old-wick} to the original equation \eqref{eq:wick-nlw}, no $U\neq 0$ can be a solution to both equations, and so $U_m\neq U$, $\mu^1$-almost surely.  \\\\ For the second claim, it follows from the arguments in \cite[Section~6]{OhPocovnicuTzvetkov}, {\cite[Theorem~1.1 and Lemma~2.7]{oh2020remark}}, see also the discussion in \cite[Section 1.3]{OOT}, that the solutions $U_m$ to \eqref{eq:shifted-old-wick} starting from $X\sim \mu^1$ converge to 0 in probability in the topology of $H^{-\eta}([-T,T];H^s)$ as $m\to \infty$, for any $s<0$ and $\eta>0$, and hence in probability in $H^{-1-\eta}([-T,T];\cH^s)$. In our setting, the large mass $m$ plays the same r\^ole in \eqref{eq:shifted-old-wick} as the divergent constant $\sigma_{1,N}\to \infty$ which must be subtracted to ensure the convergence $H_3(P_NX,{\sigma_{1,N}})\to :X^3:$. \\\\ In particular, there exists a sequence $m_j\to \infty$ such that $\|U_{m_j}\|_{{H^{-\eta-1}([-T,T];\cH^s)}}\to 0$ almost surely {for every $s<0$ and $\eta>0$}. Applying a diagonal argument to \eqref{eq:convergence_fixed_m}, there exists a sequence $j_N\to \infty$ sufficiently slowly that $U^{N,2}:=\Gamma_T(\mathfrak Z^{N,m_{j_N}})\to 0$ almost surely {in the same space}, and so that $Q^{N,m_{j_N}}\to 0$ almost surely in $\cH^s$ for every $s<0$. Letting $W^{N,2}:=Q^{N,m_{j_N}}$, ${U^{N,2}}$ is the canonical solution to the original equation \eqref{eq:wick-nlw} starting at $X^{N,2}=X+W^{N,2}$. \\ \\ It also follows that $U^{N,2}$ has, almost surely, no convergent subsequences in $\cX^s_\delta$ for any $\delta>0, s<0$. Indeed, the only possible limit point would be the zero function, whereas almost surely \[\|U^{N,2}\|_{\cX^s_\delta}\ge \|X\|_{\cH^s}-\|Q^{N,m_{j_N}}\|_{\cH^s} \to \|X\|_{\cH^s}>0\] and the proof is complete.\end{proof}

It now only remains to prove

\begin{lemma}\label{lem:shifted-enhancement-global}
Let $X\sim\mu^1$, and let
$ \mathfrak Z
    =
    \big((z,\partial_tz),z,:z^2:,:z^3:\big)$
be the associated enhanced data. Then, for every $m\geq0$ and every
$T<\infty$, $\mathcal T_m\mathfrak Z\in\mathscr D_T$ {almost surely}. Moreover, $ \Gamma_T(\mathcal T_m\mathfrak Z)
$ is a global solution on $[-T,T]$ of \eqref{eq:shifted-old-wick}.

\end{lemma}

\begin{proof}
{Write $z_0:=z$ and $\mathfrak Z_0:=\mathfrak Z$.} For $a\geq0$, set
\[
    \langle\nabla\rangle_a=(1+a-\Delta)^{1/2},
\]
and let $\gamma_a$ denote the centred Gaussian measure whose first
component has covariance $(1+a-\Delta)^{-1}$ and whose second component
is spatial white noise. For any $a\ge 0$, we will denote with a subscript $_{(a)}$ the Wick ordering with respect to the covariance of $\gamma_a$, so that the Wick ordering in the statement \eqref{eq:shifted-old-wick} is $:u^3:_{(0)}$.  Define
\[
    \sigma_{a,N}
    =
    \sum_{|n|\leq N}\frac{1}{1+a+|n|^2}; \qquad c(a)
    :=
    \lim_{N\to\infty}(\sigma_{0,N}-\sigma_{a,N}).
\]
The limit is finite, continuous and increasing, as the summand is $\mathcal{O}(\langle n\rangle^{-4})$, locally uniformly in $a$. Consequently, for every $m\geq0$, there exists a unique $a=a(m)$ such that
$
    m=a+3c(a).$ The limiting procedure defining the Wick powers produces $:u^3:_{(a)}=:u^3:_{(0)}+3c(a)u$, and so
\begin{equation}\label{eq:change-wick-mass}
    :u^3:_{(a)}+au
    =
    :u^3:_{(0)}+mu.
\end{equation}

We next compare the corresponding Gibbs measures. Letting $\mu_a$ denote the
defocusing Gibbs measure obtained from $\gamma_a$ by the density
\[
    \frac{d\mu_a}{d\gamma_a}
    =
    Z_a^{-1}
    \exp\left(
        -\frac14\int_{\mathbb T^2}:u^4:_{(a)}\,dx
    \right)
\]
so that $\mu_0=\mu^1$ is the Gibbs measure in the statement. A simple calculation using the Feldman-H{\'a}jek criterion and the estimate $$\left|\frac{\lambda_a(n)}{\lambda_0(n)}-1\right|^2{\lesssim_a} \langle n\rangle^{-4}$$ for the eigenvalues $\lambda_a(n)=(\langle n\rangle^2+a)^{-1}$ shows that Gaussian measures $\gamma_a\sim \gamma_0$ are mutually absolutely
continuous. By the construction of the Gibbs measures, $\mu_a\sim \gamma_a, \mu_0\sim \gamma_0$, and together, it follows that
\begin{equation}\label{eq:gibbs-mass-equivalence}
    \mu_a\sim\mu_0=\mu^1.
\end{equation}

The global existence theory of \cite[Theorem 1.5]{OhThomann} shows that, for the Wick-ordered wave equation with linear operator $1+a-\Delta$, there is a set
$\Omega_a$ of full $\mu_a$-measure such that, for $X\in\Omega_a$, the
canonical solution
\[
    u_a=z_a+w_a,
    \qquad
    z_a=S_a(t)X,
\]
is defined globally, where $S_a$ denotes the linear propagator associated
with $\partial_t^2+1+a-\Delta$, and where $w_a\in\mathcal X_T^{s_0}, s_0\in (\frac12,1)$ satisfies the remainder equation with data $\mathfrak{Z}_a=((z_a,\partial_t z_a), (:z^\ell_a:_{(a)}:1\le \ell\le 3))$, which simplifies to \begin{align}
 \partial_t^2w_a+(1+a-\Delta)w_a+:u_a^3:_{(a)}&=0,\label{eq:path-remainder-shifted}\\
 (w_a,\partial_tw_a)|_{t=0}&=(0,0).\label{eq: path-remainder-initial-cond-shifted}
\end{align} {We claim that, on a further full-measure event, $v:=w_a+z_a-z_0$ solves the remainder equation with data $\mathcal T_m\mathfrak Z_0$ on $[-T,T]$, belongs to $\cX_T^{s_0}$, and satisfies $\Gamma_T(\mathcal T_m\mathfrak Z_0)=u_a$. Since $\mu_a\sim\mu^1$ by \eqref{eq:gibbs-mass-equivalence}, this will complete the proof.} \\ \\  We first consider the difference of the linear components $(z_a-z_0)(t)=(S_a-S_0)(t)X$.  Choosing $0<\eta<1-s_0$, we have $
    X\in\mathcal H^{-\eta}$ almost surely. Writing $\omega_a(n)=(1+a+|n|^2)^{1/2}$, the estimate $|\omega_a(n)-\omega_0(n)|
    \lesssim_a \langle n\rangle^{-1}$ shows that \begin{equation}\label{eq:propagator-smoothing} (S_a-S_0): \cH^{-\eta}\to \cX^{1-\eta}_T \end{equation} continuously. In particular,
\[
    r_a:=z_a-z_0=(S_a-S_0)(t)X \in \cX_T^{1-\eta}\subset\mathcal X_T^{s_0}; \qquad r_a(0)=0,
    \qquad
    \partial_tr_a(0)=0.
\]
Writing $v:=w_a+r_a$, it follows that the globally defined path $u_a$ satisfies \begin{equation}\label{eq:rewrite-old-propagator}
    u_a=z_0+v,
    \qquad
    v\in\mathcal X_T^{s_0},
    \qquad
    (v,\partial_tv)|_{t=0}=(0,0).
\end{equation}

Substituting $u_a=w_a+z_a=v+z_0$ into the remainder equation for $w_a$, using the identities \eqref{eq:change-wick-mass} and $(\partial_t^2+1-\Delta)r_a=-az_a$, $v$ solves the equation \begin{equation}\label{eq:Tm-remainder-identification}
\begin{split}
    (\partial_t^2+1-\Delta)v
    &+v^3+3z_0v^2
    +3\left(:z_0^2:_{{(0)}}+\frac m3\right)v \\
    &\qquad
    +\left(:z_0^3:_{{(0)}}+mz_0\right)=0.
\end{split}
\end{equation}
which is precisely the remainder equation associated to \[
    \mathcal T_m\mathfrak Z_0
    =
    \left(
        (z_0,\partial_tz_0),
        z_0,
        :z_0^2:_{{(0)}}+\frac m3,
        :z_0^3:_{{(0)}}+mz_0
    \right).
\]
Since $v\in\mathcal X_T^{s_0}$ on the whole interval $[-T,T]$,
it follows that $
    \mathcal T_m\mathfrak Z\in\mathscr D_T
$
and, by uniqueness of the remainder equation,
$\Gamma_T(\mathcal T_m\mathfrak Z)=u_a$.

\end{proof}

\noindent{\bf  Acknowledgements}. This work was supported by the grant `ToMaBOLD' from the Norwegian Research Council (Project number 355619). The author thanks Bj\"orn Bringmann for encouraging conversations regarding the area, and to Avi Mayorcas and Ben Gess for interesting conversations in the course of preparing the manuscript.

\bibliography{literature_aihp}
\bibliographystyle{plain}

\end{document}